\documentclass[11pt]{article}

\usepackage{fullpage}
\usepackage{subcaption}

\usepackage{enumerate}
\usepackage{amsmath,amsfonts,amssymb,amsthm, bbold}
\usepackage{graphics,esint}
\usepackage{epsfig}
\usepackage{xcolor}
\usepackage{xspace}
\definecolor{pingreen}{rgb}{0,39,14}

\usepackage{xfrac}
\usepackage{comment}
\usepackage{hyperref}
\usepackage{cleveref}

\crefname{section}{§}{§§}
\Crefname{section}{§}{§§}

\def\dist{\mathrm{dist}}

\def\diam{\mathrm{diam}}

\makeatletter
\newtheorem*{rep@theorem}{\rep@title}
\newcommand{\newreptheorem}[2]{%
\newenvironment{rep#1}[1]{%
 \def\rep@title{#2 \ref{##1}}%
 \begin{rep@theorem}}%
 {\end{rep@theorem}}}
\makeatother

\newtheorem{theorem}{Theorem}[section]

\newtheorem{lemma}[theorem]{Lemma}
\newtheorem{definition}[theorem]{Definition}

\newtheorem{remark}[theorem]{Remark}

\newtheorem{proposition}[theorem]{Proposition}

\newtheorem{corollary}[theorem]{Corollary}

\newreptheorem{theorem}{Theorem}

\newcommand{\R}{\mathbb{R}}

\def\N{\mathbb N}
\def\eps{\varepsilon}

\def\eps{\varepsilon}

\def\d {\,\mathrm {d}}

\def\dt{\,\mathrm {d}t}

\def\supp{\mathrm{supp}}

\def\Id{\mathrm{Id}}

\def\div{\mathrm{div\,}}

\def\loc{\mathrm{loc}}

\numberwithin{equation}{section}

\usepackage{authblk}

\author[1]{Stefano Bianchini\thanks{bianchin@sissa.it}}
\author[2]{Sara Daneri\thanks{sara.daneri@gssi.it}}

\affil[1]{Scuola Internazionale Superiore di Studi Avanzati, Trieste, Italy}
\affil[2]{Gran Sasso Science Institute, L'Aquila, Italy}

\title{On the sticky particle solutions to the multi-dimensional pressureless Euler equations}

\date{}

\begin{document}
	\maketitle

	\begin{abstract}
		In this paper we consider the multi-dimensional pressureless Euler system and we tackle the problem of existence and uniqueness of sticky particle solutions for general measure-type initial data. Although explicit counterexamples to both existence and uniqueness are known since \cite{BN}, the problem of whether one can still find sticky particle solutions for a large set of data and of how one can select them  was up to our knowledge still completely open. 
		
		In this paper we prove that for a comeager set of initial data in the weak topology the pressureless Euler system admits a unique sticky particle solution given by a free flow where trajectories are disjoint straight lines. 
		
		Indeed, such an existence and uniqueness result holds for a broader class of solutions decreasing their kinetic energy, which we call dissipative solutions, and which turns out to be the compact weak closure of the classical sticky particle solutions. Therefore any scheme for which the energy is l.s.c. and is dissipated  will converge, for a comeager set of data, to our solution, i.e. the free flow. 
	\end{abstract}
	
	\section{Introduction}
	We consider the pressureless Euler system in $[0,T]\times\R^d$
	\begin{equation}\label{eq:E}
	\left\{\begin{aligned}
	&\partial_t\rho+\div(\rho v)=0 \\
	&\partial_t(\rho v)+\div(\rho v\otimes v)=0,
\end{aligned}\right.
	\end{equation}
	where $\rho$ is the distribution of particles and $v$ is their velocity. 

Such a model has been proposed by Zeldovich \cite{Zel} as a simplified model for the early stages of the formation of galaxies, when a dust of particles moving without pressure should start to collide and aggregate into bigger and bigger clusters. 

 Since then, several authors devoted attention to the search of sticky particle solutions, namely solutions to \eqref{eq:E} which satisfy the following adhesion principle: if two particles of fluid do not interact, then they move freely keeping constant velocity, otherwise they join with velocity given by the balance of momentum. 

The great majority of the results in the literature (see e.g. \cite{BGSW, BG, CSW1, ERS, Gre, HW, NS, NT1}) are concerned with the one-dimensional pressureless dynamics.
In this case, exploiting the density of finite particle solutions,  one can obtain from quite general initial data a global measure  solution of \eqref{eq:E} satisfying a suitable entropy condition  (see \cite{ERS} and independently \cite{Gre}). For a different approach see also \cite{BJ, PR}.  Improvements of this results regarding uniqueness of solutions were given among others in \cite{BG, HW, NT1}.   In \cite{BG} an equivalent formulation of \eqref{eq:E} is given, proving that the cumulative distribution function of the particle density is the entropy solution of a scalar conservation law.  In \cite{NT1} it is shown that the  velocity field of such solution satisfies the Oleinik condition.  In \cite{HW} uniqueness of solutions for Radon measure initial data is shown.   In \cite{NS} the authors give an alternative characterization of the evolution of \eqref{eq:E} observing that the monotone rearrangement $X_t$ of $\rho_t$ satisfies $X_t=P_{K}(X_0+tv_0)$ and $v_t=\dot X_t$, where $P_K$ is the projection operator on the cone $K$ of monotone maps. This allows the authors to investigate finer properties of solutions, and in particular their connections with gradient flows in Wasserstein spaces. A different more direct proof of the equivalence of the formulation introduced in \cite{NS} has been given in \cite{CSW1}. These approaches show that the velocity of particles is uniquely determined and the sticky particle condition is satisfied. See also \cite{Bo,Sob} for viscous approximations of \eqref{eq:E} and \cite{NT2} for a study of \eqref{eq:E} with an additional viscosity.

In general dimension, much less is known. For initial data given by a finite number of particle pointing each in a given direction, it is easy to show that a global sticky particle solution always exists and is unique.
However, in dimension $d\geq2$, one sees immediately already from a finite number of particles that the sticky particle solutions do not depend continuously on the initial data.

In \cite{BN} it is shown that, in general, both existence and uniqueness might fail: it is indeed possible to build initial data of non-existence or non-uniqueness for the sticky particle solutions, in contrast to what had been erroneously stated in \cite{Sev}. In particular, in dimension $d\geq2$ one cannot hope for a well-posedness of the Cauchy problem in the set of sticky particle solutions for all measure-type initial data as in the one-dimensional case. 

In \cite{CSW2} measure valued solutions to \eqref{eq:E} on a compactification of the state space   have been constructed for general initial data as limits of variational in time discretizations. Such solutions dissipate the total kinetic energy, and the approximating maps are constructed in the spirit of the adhesive dynamics, but no sticky particle property for the limit measure valued solutions is given.  

Thus the natural question of whether one can still find particle solutions for a large class of data (hence excluding the counterexamples in \cite{BN}) remained unanswered. In this paper we give a positive answer to this question.

In order to state our main result, define 
	\[
\mathcal P_{2,1}(\R^d\times\R^d):=\Bigl\{\nu_0\in\mathcal P(\R^d\times\R^d):\,\int|x|^2{\mathtt{p}_x}_{\#}\nu_0\leq1,\,\int|v|^2{\mathtt{p}_v}_{\#}\nu_0\leq1\Bigr\},
\]
where $\mathcal P(\R^d\times\R^d)$ are the probability measures on $\R^d\times\R^d$, $(x,v)\in\R^d\times\R^d$ the position-velocity coordinates and $\mathtt {p}_x$ $(\mathtt {p}_v)$ the projection operators on the first (last) $d$ coordinates. Moreover, we consider the problem of existence and uniqueness in a larger class of solutions which we call \emph{dissipative} since in particular their kinetic energy is decreasing but their trajectories might cross without joining at later times. By \emph{free flow} we mean a flow in which trajectories are disjoint straight lines which never intersect.  

Our main result is the following:

\begin{theorem}\label{thm:freeflow0}

	There is a set $D_0\subset\mathcal P_{2,1}(\R^d\times\R^d)$ such that, for any $\nu_0\in D_0$ there exists a unique dissipative  solution $\eta$ with initial data $\nu_0$ and it is given by a free flow. Such a set is a dense $G_{\delta}$ set (i.e. of second category)  in the weak topology on $\mathcal P_{2,1}(\R^d\times\R^d)$.
\end{theorem}
Since our notion of dissipative solution includes the classical sticky particle solutions, the above theorem implies that, even though the sticky particle solutions are not well-posed for every measure-type initial data, there exists a comeager set of initial data in the weak topology giving rise to a unique sticky particle solution. Moreover, for any of these initial data the sticky particle  solution is unique also in the larger class of dissipative solutions (where trajectories are allowed to cross) and is given 
by a trivial free flow concentrated on trajectories which do not intersect. In particular for such initial data there is only one dissipative solution and its dissipation is equal to zero. Thus, for a comeager set of initial data the problem of finding sticky particle solutions is well-posed, but the dynamics that one  sees is trivial.

Both the concepts of dissipative and classical sticky particle solutions are defined at a Lagrangian level as measures on the space of curves with finite energy. The class of dissipative solutions turns out to be the compact weak closure of the set of classical sticky particle solutions (see Theorem  \ref{thm:stickydense}).In Section \ref{sec:diss} we introduce the concept of dissipative solution and we show that this class is compact and includes the classical sticky particle solutions.  
Then in Section \ref{sec:approx} we show that dissipative solutions can be approximated in the weak topology by classical sticky particle solutions. In Section \ref{S:PDE_formul} we give a kinetic and PDE formulation of our notion of dissipative solution. Given the approximation result of Section \ref{sec:approx}, in Section \ref{sec:gdelta} we use the fact that the fact that the dimension is greater than or equal to $2$ to modify the initial data of such finite particle solutions in order to have the trivial free flow as unique solution, while staying close in the weak topology. Such data in particular have the property that every dissipating solution starting from them has zero dissipation. The fact that such initial data are a $G_\delta$ set follows from the compactness of the set of dissipative solutions and the upper semicontinuity of the dissipation.  \\
The construction of this dense $G_\delta$-set relies on some natural assumptions on the approximation scheme, in particular that the energy is l.s.c. and that if the dissipation of energy is $0$ then the only solution is the free flow, see Remark \ref{Rem:other_G_delta} for details. Hence one concludes that for a dense $G_\delta$-set of initial data the weak solutions constructed by any reasonable approximation scheme coincide with our dissipative solutions, i.e. the free flow. 


	\section{Preliminaries} 
	
	We define the following space of curves. For a fixed $T>0$,
	\begin{equation*}
	\Gamma := \Big\{ \gamma \in L^2 \big( (-1,T), \R^d \big):\, \gamma \llcorner_{(-1,0)} \ \text{affine} \Big\}.
	\end{equation*}
	
	\begin{remark}
	\label{Rem:other_possible_spaces}
	The above choice is in order to avoid assigning the initial speed $W_0(\gamma)$ and considering the space ${L^2}((0,T);\R^d)) \times \R^d$ with the product topology. The solutions we consider will actually be in $W^{1,2}((-1,T),\R^d))$ (see Lemma \ref{Lem:calM_right} below): of course the map
	\begin{equation*}
	\Gamma \ni \gamma \mapsto \big( \gamma \llcorner_{(0,T)},\gamma(0) - \gamma(-1) \big) \in W^{1,2}((0,T);\R^d) \times \R^d
	\end{equation*}
	is a bijection, with inverse
	\begin{equation*}
	W^{1,2}((0,T);\R^d) \times \R^d \ni (\gamma,W_0) \mapsto \begin{cases}
	\lim_{s \searrow 0} \gamma(s) - t W_0 & t \in (-1,0], \\
	\gamma(t) & 0 < t < T.
	\end{cases}
	\end{equation*}
	Actually these maps are bounded linear operators {when we consider the weak or strong topology of $W^{1,2}$}.
	\end{remark}
	
	
	For every $\gamma\in\Gamma$, we define the initial velocity field as
	
	\begin{equation}\label{eq:v0}
	W_0(\gamma) = \gamma(0) - \gamma(-1).
	\end{equation}
	
	This function is continuous in the topology of $\Gamma$.
	
	We denote by $\mathcal P(\Gamma)$ the set of  Borel probability measures on $\Gamma$.
	
	On $\mathcal P(\Gamma)$ we consider the topology induced by the narrow convergence, namely $\eta^n\rightharpoonup\eta$ in $\mathcal P(\Gamma)$ if $\int\phi(\gamma)\eta^n(d\gamma)\to\int\phi(\gamma)\eta(d\gamma)$ for any bounded $\phi\in C^0(\Gamma)$. 
	
    We define the following closed subset of $\mathcal P(\Gamma)$
	\begin{equation*}
	\mathcal M (\Gamma) = \bigg\{ \eta \in \mathcal P(\Gamma) : \int |\gamma(0)|^2 \eta(d\gamma)\leq 1,\, \int \|\dot \gamma\|_{L^2(-1,T)}^2 \eta(d\gamma)  \leq 1 \bigg\}. 
	\end{equation*}
	{Note that $\gamma(0)$ is defined because $\gamma \in W^{1,2}$ $\eta$-a.e.}
	
	The set $\mathcal M(\Gamma)$ satisfies the following properties.
	\begin{lemma}
		\label{Lem:calM_right}
		The set $\mathcal M(\Gamma)$ is tight. 
	\end{lemma}
	
	\begin{proof}
		For each $n\in\N$, the set
		\begin{equation*}
		\Gamma(n):=\Big\{ \gamma : |\gamma(0)|^2\leq n, \|\dot \gamma\|^2_{L^2(-1,T)} \leq n \Big\}
		\end{equation*}
		is bounded and therefore compact in $(\Gamma,d_{\Gamma})$. {In particular the set of bounded measures supported on $\Gamma(n)$ is compact in the narrow topology}.
		
		Now notice that $\forall\, \eta \in \mathcal M(\Gamma)$ by Chebyshev's inequality
		\begin{equation*}
		\max\bigg\{\eta \bigg( \bigg\{ \gamma : |\gamma(0)|^2>n\bigg\}\bigg),\, \eta \bigg( \bigg\{ \gamma :\int_{-1}^T |\dot \gamma(t)|^2 dt > n \bigg\} \bigg)\bigg\} < \frac{1}{n},
		\end{equation*}
		so that $\mathcal M(\Gamma)$ is tight.
%
	\end{proof}

	\begin{corollary}
		\label{Cor:calM_compact_metric}
		The space $\mathcal M(\Gamma)$ with the topology of $\mathcal P(\Gamma)$  is compact metrizable.
	\end{corollary}
	
	\begin{proof}
		For the restrictions of the measures in $\mathcal M(\Gamma)$ to the compact sets $(\Gamma(n), d_{\Gamma})$, $n\in\N$, metrize the narrow convergence with the usual L\`evy Prokorhov distance $\hat d_n$. Then the distance metrizing the narrow convergence between two measures in $\mathcal M(\Gamma)$ is defined by
		\[
		d_{\mathcal M(\Gamma)}(\eta_1,\eta_2)=\sum_n2^{-n}\hat d_n(\eta_1\llcorner_{\Gamma(n)},\eta_2\llcorner_{\Gamma(n)}). \qedhere
		\]
	\end{proof}

Being concentrated on $W^{1,2}((-1,T),\R^d)$, the measures $\eta \in \mathcal M(\Gamma)$ satisfy the following property.

\begin{lemma}
\label{Lem:conti_weak_cont}
Let $\phi : W^{1,2}((-1,T),\R^d) \to \R$ is bounded and continuous w.r.t. the weak topology of $W^{1,2}((-1,T),\R^d)$, i.e. the $L^2$-topology on $\gamma$ and the weak topology on $\dot \gamma$. Then if $\eta_n \rightharpoonup \eta$ narrowly, then
\begin{equation*}
\int \phi(\gamma) \eta_n(d\gamma) \to \int \phi(\gamma) \eta(d\gamma).
\end{equation*}
\end{lemma}

\begin{proof}
Since $\Gamma(m)$ is compact in both topologies, it follows that the $L^2$-metric and the metrization of the weak topology of $W^{1,2}((-1,T),\R^d)$ are equivalent. Let $\tilde \phi_m(\gamma)$ be an $L^2$-continuous extension of $\phi \llcorner_{\Gamma(m)}$ to $L^2$, with the same bound as $\phi$. Hence
\begin{equation*}
\int \tilde \phi_m \eta_n \to \int \tilde \phi_m \eta,
\end{equation*}
and by tightness of $\mathcal M(\Gamma)$ we conclude.
\end{proof}

	\section{Dissipative  solutions}\label{sec:diss}
	
	In this section we first give a definition of dissipative solutions to the system \eqref{eq:E} in the Lagrangian formulation, namely as a subset of $\mathcal M(\Gamma)$. Such solutions dissipate the total kinetic energy, but trajectories are allowed to intersect at a certain time without joining in the future. The so-called sticky particle (or adhesive) solutions constitute a subset of dissipative solutions whose trajectories, whenever intersecting at some time, must coincide for all subsequent times.

	\subsection{Definition of dissipative solution}
	\label{Ss:definit_sticky}
	
	For any $t\in(0,T)$, define the space of curves
	\[
	\Gamma_t := {L^{2}} \big( (t,T), \R^d \big).
	\]
	
	On $\Gamma_t$ we put the {$L^2$-metric}. 

	Let $T_t:\Gamma\to \Gamma_t$ be the restriction map
	
	\[
	T_t(\gamma)=\gamma\llcorner_{(t,T)}.
	\]
	
	{This map is a contraction.}
	The map $T_t$ induces the following equivalence relation on $\Gamma$
	\[
	\gamma\sim_t\gamma'\quad {\Longleftrightarrow} \quad T_t(\gamma)=T_t(\gamma'). 
	\]
     
    and the corresponding disintegration of measures $\eta\in\mathcal M(\Gamma)$ 
    \begin{equation}
    \label{eq:disinte_eta_at_time_t}
	\eta = \int \omega^t_{\gamma'} {T_t}_{\#}\eta(d\gamma'), \qquad {T_t}_{\#}\eta = (T_t)_\sharp \eta, \ \gamma'\in\Gamma_t,
	\end{equation}
	
	where $\omega^t_{\gamma'}\in\mathcal M(\Gamma)$ satisfies $\omega^t_{\gamma'}(T_t^{-1}(\gamma'))=1$, namely it is strongly consistent according to the notation used in \cite{Fr}.
	
		Observe the following
	
	\begin{lemma}
		\label{Lem:dot_gamma_eq_proje}
		It holds for $\mathcal L^1 \times \eta$-a.e. $(t,\gamma)\in(-1,T)\times\Gamma$
		\begin{equation*}
		\dot \gamma(t) = \int \dot \gamma'(t) \omega^t_{T_t(\gamma)}(d\gamma').
		\end{equation*}
	\end{lemma}
	
	\begin{proof}
		Let $t\in(-1,T)$ be such that $\dot\gamma(t)$ exists for $\eta$-a.e. $\gamma$.
		If $T_t(\gamma')=T_t(\gamma)$ and $\gamma'$ is differentiable in $t$, then $\dot\gamma(t)=\dot\gamma'(t)$. Thus removing a set of $\mathcal L^1\times\eta$-measure $0$, it follows that $\dot \gamma( t)$ is constant on equivalence classes, from which the above formula follows.
	\end{proof}

Moreover, one has the following

	\begin{lemma}
	The map
	\begin{equation*}
	\eta \mapsto \eta_t = (T_t)_\sharp \eta
	\end{equation*}
	is continuous from $(\mathcal M(\Gamma), d_{\mathcal M(\Gamma)})$ to $(\mathcal M(\Gamma_t), d_{\mathcal M(\Gamma_t)})$.
\end{lemma}

\begin{proof}
	The map $T_t$ is a {contraction}. Hence for any continuous function $\phi$ on $\Gamma_t$ the map $\phi\circ T_t$ is continuous on $\Gamma$ and the statement of the lemma follows from the definition of push-forward.
\end{proof}
	
	Recalling the definition of $W_0$ given in \eqref{eq:v0}, define for all $t\in[0,T]$
	
	\begin{equation}
	\label{eq:sticky_particle_property}
	V_t(\gamma)=\int W_0(\gamma') \omega^t_{T_t(\gamma)}(d\gamma').
	\end{equation}
	
	{This function is defined for $\eta$-a.e. $(t,\gamma)$.}
	
	\begin{remark}
		Notice that in general it may happen that $V_0(\gamma)\neq W_0(\gamma)$, and therefore the velocity of the curves in $\Gamma$ we consider may have an initial jump at time $t=0$. Since, as we will see in Corollary \ref{cor:BV}, for the kind of solutions $\eta\in\mathcal M(\Gamma)$ we consider the map $t\mapsto V_t$ defined from a vector field $W_0$ will be right continuous in $L^2_{\eta}$,  one could equivalently modify the curves $\gamma$ on $(-1,0)$ (and therefore $W_0$) defining   $\dot\gamma_{\llcorner(-1,0]}=\lim _{t\searrow 0}V_t(\gamma)$.
	\end{remark}

	We are now ready to give our definition of dissipative solution. 
	\begin{definition}
		\label{def:sticky}
		We say that $\eta\in\mathcal M(\Gamma)$ is a \emph{dissipative solution} of \eqref{eq:E} if it holds
		\begin{equation}\label{eq:diss}
		\dot \gamma(t) = V_t(\gamma) \qquad \mathcal L^1 \times \eta\text{-a.e. on $[0,T]\times \Gamma$.}
		\end{equation}
	\end{definition}
	
	The above condition can be rewritten as follows: for every continuous bounded function $\phi(t,\gamma)$ on $[0,T]\times\Gamma$ it holds
	\begin{equation*}
	\int \bigg[ \int \phi(t,\gamma) \dot \gamma(t) dt \bigg] \eta(d\gamma) = \int \bigg[ \int \phi(t,\gamma) V_t(\gamma) \eta(d\gamma) \bigg] dt.
	\end{equation*}
	
The fact that, for $t>s$, the map $T_t$ induces a coarser partition than $T_s$ (or, in other words, that it induces a descending in time  filtration on $\Gamma$), can be expressed at the level of disintegrations of $\eta$ at different times in the following way. 

For $s<t$, let $T_{s\to t}:\Gamma_s\to\Gamma_t$ be the restriction map such that $T_{s\to t}\circ T_s=T_t$. Then {by disintegrating again}

\[
{(T_s)}_{\#}\eta(d\gamma')=\int\omega^{s\to t}_{\gamma''}(d\gamma'){(T_t)}_{\#}\eta(d\gamma'')
\]

and 

\begin{align}
\eta(d\gamma)&=\int\omega^s_{\gamma'}(d\gamma){(T_s)}_{\#}\eta(d\gamma')\notag\\
&=\int\int\omega^s_{\gamma'}(d\gamma)\omega^{s\to t}_{\gamma''}(d\gamma')(T_{s\to t}\circ T_s)_{\#}\eta(d\gamma'')\notag\\
&=\int\int\omega^s_{\gamma'}(d\gamma)\omega^{s\to t}_{\gamma''}(d\gamma'){(T_t)}_{\#}\eta(d\gamma'')\notag\\
&=\int\omega^t_{\gamma''}(d\gamma){(T_t)}_{\#}\eta(d\gamma'').
\end{align}

Therefore, by the uniqueness of the disintegration, {for $(T_t)_\sharp \eta$-a.e. $\gamma''$}
\begin{equation}\label{eq:rhoomega}
\omega^t_{\gamma''}(d\gamma)=\int\omega^s_{\gamma'}(d\gamma)\omega^{s\to t}_{\gamma''}(d\gamma').
\end{equation}

The terminology used for this kind of solutions is consistent with the following

\begin{proposition}\label{prop:diss}
	Let $\Psi:\R^d\to\R$ be a convex function and let $\eta\in\mathcal M(\Gamma)$ be a dissipative solution. Then, the map
	\begin{equation}\label{eq:psidiss}
	t\mapsto\int\Psi(V_t(\gamma))\eta(d\gamma)
	\end{equation}
	is nonincreasing on $(-1,T]$, where by convention we set $V_t=W_0$ if $t\in (-1,0)$. In particular, taking $\Psi=|\cdot|^2$ one has the dissipation balance
	\begin{equation}\label{eq:squarediss}
	\int|V_s(\gamma)-V_t(\gamma)|^2\eta(d\gamma)=\int|V_s(\gamma)|^2\eta(d\gamma)-\int|V_t(\gamma)|^2\eta(d\gamma), \quad\forall\,s\leq t, s,t\in(-1,T].
	\end{equation}
\end{proposition}

\begin{proof}
	Recalling the disintegration formula \eqref{eq:rhoomega} and applying Jensen's inequality one has that
	\begin{align}
	\int\Psi(V_t(\gamma))\eta(d\gamma)&=\int\Psi\biggl(\int W_0(\gamma')\omega^t_{T_t(\gamma)}(d\gamma')\biggr){(T_t)}_{\#}\eta(d\gamma)\notag\\
	&=\int\Psi\biggl(\int W_0(\gamma')\omega^s_{\gamma''}(d\gamma')\omega^{s\to t}_{T_t(\gamma)}(d\gamma'')\biggr){(T_t)}_{\#}\eta(d\gamma)\notag\\
	&\leq \int\Psi(V_s(\gamma''))\omega^{s\to t}_{T_t(\gamma)}(d\gamma''){(T_t)}_{\#}\eta(d\gamma)\notag\\
	&=\int\Psi(V_s(\gamma)){(T_s)}_{\#}\eta(d\gamma)\notag\\
	&=\int\Psi(V_s(\gamma))\eta(d\gamma).
	\end{align}
	
	To prove that the dissipation balance \eqref{eq:squarediss} holds, we observe that
	
	\begin{align}
	\int V_t(\gamma)\cdot V_s(\gamma)\eta(d\gamma)&=\int\int V_t(\gamma)\cdot V_s(\gamma)\omega^s_{\gamma'}(d\gamma){(T_s)}_{\#}\eta(d\gamma')\notag\\
	&=\int V_t(\gamma')\cdot V_s(\gamma'){(T_s)}_{\#}\eta(d\gamma')\notag\\
	&=\int\int V_t(\gamma')\cdot V_s(\gamma')\omega^{s\to t}_{\gamma''}(d\gamma'){(T_t)}_{\#}\eta(d\gamma'')\notag\\
	&=\int V_t(\gamma'')\cdot V_t(\gamma''){(T_t)}_{\#}\eta(d\gamma'')\notag\\
	&=\int|V_t(\gamma)|^2\eta(d\gamma).
	\end{align}
\end{proof}

Let us also define the total dissipation of a measure $\eta\in\mathcal M(\Gamma)$ as follows.

\begin{definition}[Total dissipation]
	We define the total dissipation of a  measure $\eta\in\mathcal M(\Gamma)$ as
	\begin{equation}\label{eq:dissipation}
	D(\eta)=T\|W_0\|_{L^2_\eta}^2-\int\int_0^T|V_t(\gamma)|^2dt\eta(d\gamma).
	\end{equation}
\end{definition}

Notice that for any dissipative solution $\eta$, $D(\eta)\geq0$. Moreover, for any initial datum $(e_0,W_0)_{\#}\eta\in\mathcal P_2(\R^d\times\R^d)$, being $e_0:\Gamma\to\R^d$ the evaluation map $e_0(\gamma)=\gamma(0)$, there exists always a dissipative solution concentrated on straight lines of constant velocity $W_0(\gamma)$, which has zero dissipation.

We now give a precise definition of the concept of sticky particle solution. For every $t\in[0,T]$, let $e_t:\Gamma\to\R^d$ be the evaluation map $e_t(\gamma)=\gamma(t)$. 

\begin{definition}
	We say that $\eta\in\mathcal M(\Gamma)$ is \emph{sticky particle solution} of \eqref{eq:E} if $\eta$ is concentrated on a subset of $\Gamma$ on which, for all $t\in[0,T]$, the maps $T_t$ and $e_t$ induce the same equivalence relation.
\end{definition}

\begin{figure}
	Two particles which collide and stick or intersect without joining.
\end{figure}

	\subsection{Compactness of the set of dissipative solution}
	\label{Ss:conve_sticky}
	
	The aim of this section is to show that the set of dissipative solutions is closed w.r.t. weak convergence. In order to pass to the limit in the relation \eqref{eq:diss} we embed the dissipative solutions into a larger compact space of Young measures generated by the disintegrations of the measures $\eta\in\mathcal M(\Gamma)$ w.r.t. the restriction maps $T_t$.

	For any  $\eta\in\mathcal P(\Gamma)$, define the map 
	\begin{equation}
	F(\eta):[0,T]\times\Gamma\to\mathcal P(\Gamma), \quad F(\eta)(t,\gamma)=\omega^t_{T_t(\gamma)},
	\end{equation} 
	where $\{\omega^t_{\gamma'}\}_{\gamma'\in\Gamma_t}$ is the disintegration of $\eta$ w.r.t. the restriction map $T_t$. {This map is defined $\mathcal L^1 \times \eta$-a.e..}
	Let then 
	\begin{equation}\label{eq:mueta}
	\mu(\eta)=\big(\Id_{[0,T]\times\Gamma}\times F(\eta)\big)_{\#}\big(\mathcal L^1\times \eta\big).
	\end{equation} 
	By definition,
	\[
	\mu(\eta)\in \mathcal P([0,T]\times\Gamma\times\mathcal P(\Gamma)).
	\]
	
	With the disintegration \eqref{eq:disinte_eta_at_time_t} we can rewrite $\mu(\eta)$ as
	\begin{equation*}
	\int \phi(t,\gamma,\omega) \mu(\eta)(dtd\gamma d\omega) = \int \phi(t,\gamma,\omega^t_{T_t(\gamma')}) \eta(d\gamma) dt.
	\end{equation*}
	for every continuous function $\phi$.

	\begin{lemma}
		The set 
		\[
		\biggl\{\mu(\eta)\in\mathcal P([0,T]\times\Gamma\times\mathcal P(\Gamma)):\,\eta\in\mathcal M(\Gamma) \biggr\}
		\]
		is tight.
	\end{lemma}
\begin{proof}
	Define the set
	\[
	\mathcal M_n(\Gamma) = \bigg\{ \eta\in\mathcal P(\Gamma):\,\int|\gamma(0)|^2\eta(d\gamma)\leq n,\,\int\|\dot\gamma\|_2^2\eta(d\gamma)\leq n \bigg\}.
	\]
	Then the set 
	\[
	K(n)=[0,T]\times\Gamma(n)\times\mathcal M_n(\Gamma)
	\]
	is compact in $[0,T]\times\Gamma\times\mathcal P(\Gamma)$ {(just adapt the proofs of Lemma \ref{Lem:calM_right} and Corollary \ref{Cor:calM_compact_metric} to $\mathcal M_n(\Gamma)$)}.
	Let now $\eta\in\mathcal M(\Gamma)$ and $\mu(\eta)$ defined as in \eqref{eq:mueta}.
	
	On the one hand,
	\begin{equation}
\mu(\eta)([0,T]\times\Gamma(n)^c\times\mathcal P(\Gamma))=\bigl(\mathcal L^1\times \eta\bigr)([0,T]\times\Gamma(n)^c)\leq\frac{T}{n}.
\end{equation}

On the other hand, 
\begin{align}
\mu(\eta)(&[0,T]\times\Gamma(n)\times\mathcal M_n(\Gamma)^c)\leq\int_0^T \eta \big( \{\gamma\in\Gamma:\,\omega^t_{T_t(\gamma)}\in\mathcal M_n(\Gamma)^c\} \big) dt \notag\\
&=\int_0^T\eta\bigg(\bigg\{\gamma\in\Gamma:\int|\gamma'(0)|^2\omega^t_{T_t(\gamma)}(d\gamma')>n \text{ or } \int\|\dot{\gamma'}\|_2^2 \omega^t_{T_t(\gamma)}(d\gamma')>n\biggr\}\biggr) dt.
\end{align}
Since $\eta\in\mathcal M(\Gamma){= \mathcal M_1(\Gamma)}$,
\[
1\geq\int|\gamma(0)|^2\eta(d\gamma)=\int\int|\gamma'(0)|^2\omega^t_{T_t(\gamma)}(d\gamma'){(T_t)}_{\#}\eta(dT_t(\gamma)),
\]
therefore
\[
\eta\biggl(\biggl\{\gamma\in\Gamma:\,\int|\gamma'(0)|^2\omega^t_{T_t(\gamma)}(d\gamma')>n\biggr\}\biggr)<\frac1n.
\]
Similarly, one proves that 
\[
\eta\biggl(\biggl\{\gamma\in\Gamma:\,\int\|\dot{\gamma'}\|_2^2\omega^t_{T_t(\gamma)}(d\gamma')>n\biggr\}\biggr)\leq\frac{1}{n}
\]
and the lemma is proved.
\end{proof}

The closure of the set $\biggl\{\mu(\eta)\in\mathcal P([0,T]\times\Gamma\times\mathcal P(\Gamma)):\,\eta\in\mathcal M(\Gamma) \biggr\}$ is {contained in} the compact set

\begin{equation}
\begin{split}
\mathcal D=\biggl\{\mu\in\mathcal P([0,T]\times\Gamma\times\mathcal P(\Gamma))&:\,\mathit {p_{[0,T]\times\Gamma}}_{\#}\mu=\mathcal L^1\times\eta,\,\eta\in\mathcal M(\Gamma), \\
& \ \int\int \frac{|\gamma'(0)|^2}{T} \omega(d\gamma')\mu(dtd\gamma d\omega) \leq 1, \int\int \|\dot{\gamma'}\|_2^2\omega(d\gamma')\mu(dtd\gamma d\omega) \leq 1 \biggr\}.
\end{split}
\end{equation}

Indeed, the fact that each $\mu \in \mathcal D$ is tight follows from the same proof of the previous lemma. The fact that every narrow limit $\mathcal D \ni \mu_n \to \mu$ belongs to $\mathcal D$ is a consequence of the observation that \begin{equation*}
\omega \mapsto \int \|\dot{\gamma'}\|_2^2 \omega(d\gamma')
\end{equation*}
is l.s.c. w.r.t. the narrow convergence, being $\gamma' \to \|\dot{\gamma'}\|_2^2$ convex l.s.c..

Our aim is now to prove that if $\eta^n\rightharpoonup\eta$ with $\eta^n$ dissipative solutions, then also $\eta$ is dissipative. We will do it looking at the weak limit of the $\mu(\eta^n)$ in $\mathcal D$. In particular, we will use a dense family of smooth test functions of the form

\[
\phi_0(t)\tilde{\phi}(T_{\bar t}(\gamma))\phi(\gamma)\psi(\omega),
\]
where $\bar t>t$ for all $t\in\supp\phi_0$.

\begin{lemma}
	Any weak limit $\mu$ of a weakly convergent sequence $\{\mu(\eta^n):\eta^n\in\mathcal M(\Gamma),\,n\in\N\}$ satisfies
	\begin{equation}\label{eq:test1}
	\int\phi_0(t)\tilde{\phi}(T_{\bar t}(\gamma))\phi(\gamma)\psi(\omega)\mu(dtd\gamma d\omega)=\int\phi_0(t)\tilde{\phi}(T_{\bar t}(\gamma))\psi(\omega)\biggl(\int\phi(\gamma')\omega(d\gamma')\biggr)\mu(dtd\gamma d\omega).
	\end{equation}
\end{lemma}

\begin{proof}
	One has that
	\begin{align}
	\int\phi_0(t)\tilde{\phi}(T_{\bar t}(\gamma))\phi(\gamma)\psi(\omega)\mu(\eta)(dtd\gamma d\omega)&=\int\int\phi_0(t)\tilde{\phi}(T_{\bar t}(\gamma))\phi(\gamma)\psi(\omega^t_{T_t(\gamma)})dt\eta(d\gamma)\notag\\
	&=\int\phi_0(t)\biggl(\int\int\tilde{\phi}(T_{\bar t}(\gamma))\phi(\gamma)\psi(\omega^t_{T_t(\gamma)})\omega^t_{\gamma'}(d\gamma){(T_t)}_{\#}\eta(\d\gamma')\biggr)dt\notag\\
	&=\int\phi_0(t)\biggl(\int \tilde \phi(T_{t,\bar t}(\gamma'))\biggl(\int\phi(\gamma)\psi(\omega^t_{\gamma'})\omega^t_{\gamma'}(d\gamma)\biggr){(T_t)}_{\#}\eta(d\gamma')\biggr)dt\notag\\
	&=\int\phi_0(t)\int\tilde{\phi}(T_{\bar t}(\gamma))\int\phi(\gamma')\psi(\omega^t_{T_t(\gamma)})\omega^t_{T_t(\gamma)}(d\gamma')\eta(d\gamma)dt\notag\\
	&=\int\phi_0(t)\tilde{\phi}(T_{\bar t}(\gamma))\biggl(\int\phi(\gamma')\omega(d\gamma')\biggr)\psi(\omega)\mu(\eta)(dtd\gamma d\omega).
	\end{align}
Since
	\begin{equation*}
	\omega \mapsto \int \phi(\gamma') \omega(d\gamma')
	\end{equation*}
	is continuous for $\phi$ continuous, then the set of $\mu$ with the property above is closed, and thus is contains the closure of $\{\mu(\eta),\eta \in \mathcal M(\Gamma)\}$.
\end{proof}

Now, using the fact that $\phi_0$  and $\bar t$ are arbitrary {under the condition $\supp \phi_0 \subset (-\infty,\bar t)$}, from \eqref{eq:test1} it follows that, for $\mathcal L^1$-a.e. $t \in [0,T]$

\begin{equation}\label{eq:test2}
\int\tilde{\phi}(T_t(\gamma))\phi(\gamma)\biggl(\int\psi(\omega)\mu_{t,\gamma}(d\omega)\biggr)\eta(d\gamma)=\int\tilde{\phi}(T_t(\gamma))\biggl(\int\psi(\omega)\biggl(\int\phi(\gamma')\omega(d\gamma')\biggr)\mu_{t,\gamma}(d\omega)\biggr)\eta(d\gamma),
\end{equation}
where 
\[
\mu(dtd\gamma d\omega)=\int\mu_{t,\gamma}(d\omega)dt\eta(d\gamma).
\]
Indeed by letting $\phi_0$ vary, one obtains for {$\mathcal L^1$-a.e. $t$ and} all $\bar t > t$
\begin{equation*}
\int \tilde{\phi}(T_{\bar t}(\gamma)) \phi(\gamma) \biggl( \int \psi(\omega) \mu_{t,\gamma}(d\omega) \biggr) \eta(d\gamma) = \int \tilde{\phi}(T_{\bar t}(\gamma)) \biggl( \int \psi(\omega) \biggl( \int \phi(\gamma') \omega(d\gamma') \biggr) \mu_{t,\gamma}(d\omega) \biggr) \eta(d\gamma).
\end{equation*}
If $\tilde \phi : \Gamma_t \to \R$ is a continuous function, one can test the above equation with
\begin{equation*}
\tilde \phi'(T_{\bar t}(\gamma)) = \tilde \phi \big( \gamma(\bar t) \chi_{[t,\bar t)} + T_{\bar t}(\gamma)\chi_{[\bar t,T]} \big).
\end{equation*}
The above function $\tilde \phi'  : \Gamma_{\bar t} \to \R$ is continuous because of the embedding $W^{1,2}([0,T])$ in $C^0$, and as $\bar t \searrow t$  the function $\tilde{\phi}'\circ T_{\bar t}$ converges pointwise to $\tilde \phi\circ T_t$. By the bounded convergence theorem we conclude that \eqref{eq:test2} holds.

Replacing $\eta$ in \eqref{eq:test2} with its disintegration w.r.t. $T_t$ and thanks to the fact that $\tilde{\phi}$ is arbitrary, one obtains that, for ${T_t}_{\#}\eta$-a.e. $\gamma'\in\Gamma_t$

\begin{align}
\int\phi(\gamma)\biggl(\int\psi(\omega)\mu_{t,\gamma}(d\omega)\biggr)\omega^t_{\gamma'}(d\gamma)=\int\biggl(\int\psi(\omega)\biggl(\int\phi(\gamma')\omega(d\gamma')\biggr)\mu_{t,\gamma}(d\omega)\biggr)\omega^t_{\gamma'}(d\gamma).
\end{align} 	
Taking now $\psi=1$, one gets
\begin{equation}
\int \phi(\gamma)\omega^t_{\gamma'}(d\gamma)=\int\int\biggl(\int\phi(\gamma')\omega(d\gamma')\biggr)\mu_{t,\gamma}(d\omega)\omega^t_{\gamma'}(d\gamma)
\end{equation}
and by the arbitrariness of $\phi$ 
\begin{equation}\label{eq:test3}
\omega^t_{\gamma'}(d\gamma)=\int\omega(d\gamma)\mu_{t,\gamma''}(d\omega)\omega^t_{\gamma'}(d\gamma'').
\end{equation}

We will use condition \eqref{eq:test3} to show the following proposition. 

\begin{proposition}
Let $\{\eta^n\}\subset\mathcal M(\Gamma)$ be a sequence of dissipative solutions such that $\eta^n\rightharpoonup\eta\in\mathcal M(\Gamma)$ as $n\to\infty$.
Then $\eta$ is a dissipative solution.	
\end{proposition}
	
	\begin{proof}
		Let $\phi:[0,T]\times\Gamma\to\R$ be a continuous function. 
		
		On the one hand, as $n\to\infty$
		\[
		\int\phi(t,\gamma)\dot{\gamma}(t)\eta^n(d\gamma)dt\to\int\phi(t,\gamma)\dot{\gamma}(t)\eta(d\gamma)dt.
		\]
		
	Indeed, $t \mapsto \int \phi(t,\gamma) \dot \gamma(t)dt$ is continuous in $\Gamma$ {w.r.t. the weak convergence, and then one applies Lemma \ref{Lem:conti_weak_cont}.}
		
		On the other hand,
		
	\begin{align}
		\int\phi(t,\gamma)\dot{\gamma}(t)\eta^n(d\gamma)\dt&=\int\int\phi(t,\gamma)\int W_0(\gamma')\omega^{n,t}_{T_t(\gamma)}(d\gamma')\eta^n(d\gamma)dt\notag\\
		&=\int\int\phi(t,\gamma)\int W_0(\gamma')\omega(d\gamma')\mu(\eta^n)_{t,\gamma}(d\omega)\eta^n(d\gamma)dt\notag\\
		&\to\int\int\phi(t,\gamma)\int W_0(\gamma')\omega(d\gamma')\mu_{t,\gamma}(d\omega)\eta(d\gamma)dt \quad\text{ as $n\to\infty$}.
	\end{align}
	Hence,
	\begin{equation}\label{eq:test4}
	\dot{\gamma}(t)=\int\int W_0(\gamma')\omega(d\gamma')\mu_{t,\gamma}(d\omega)\quad\mathcal L^1\times\eta{\text{-a.e. $(t,\gamma)$}}.
	\end{equation}
	Integrating \eqref{eq:test4} w.r.t. $\omega^t_{\gamma''}$ and then  using \eqref{eq:test3} one gets
	\begin{align}
	\int	\dot{\gamma}(t)\omega^t_{\gamma''}(d\gamma)&=\int\int\int W_0(\gamma')\omega(d\gamma')\mu_{t,\gamma}(d\omega)\omega^t_{\gamma''}(d\gamma)\notag\\
	&=\int W_0(\gamma)\omega^t_{\gamma''}(d\gamma).
	\end{align}
	Finally, since $\dot{\gamma}(t)$ is constant on $T_t^{-1}(\gamma'')$, we deduce that
	\[
	\dot{\gamma}(t)=\int W_0(\gamma')\omega^t_{T_t(\gamma)}(d\gamma') \quad\mathcal L^1\times\eta\text{-a.e.}\ (t,\gamma),
	\]
	namely that $\eta$ is a dissipative solution.
	\end{proof}

We have therefore proved the following 

\begin{theorem}\label{thm:sticky}
	The set of dissipative solutions in $\mathcal M(\Gamma)$ is compact.
\end{theorem}

\begin{remark}
	Notice instead that the set of sticky particle solutions in $\mathcal M{(\Gamma)}$ is not closed. Take for example the free flow of two particles which do not interact up to time $T$ and change their directions so that in the limit their trajectories intersect for some time $t\in(0,T)$. Or, even worse, see the Example 4 in \cite{BN}, where the only solution is the trivial one {(i.e. the free flow)} with zero dissipation (not sticky), as limit of sticky particle solutions for the initial data obtained removing the particles in a smaller and smaller neighbourhood of the origin. 
	
	Therefore {it is justified} the necessity to consider the notion of dissipative solution for general initial data. 
\end{remark}

\subsection{Properties of dissipative solutions}

\begin{corollary}\label{cor:BV}
Let $\eta\in\mathcal M(\Gamma)$ be a dissipative solution. Then,	the map $t\mapsto V_t$ belongs to $BV^{1/2}([0,T];L^2_\eta(\Gamma;\R^d))$ and it is right continuous.

\end{corollary}	

\begin{proof}
	Let $0=t_0\leq t_1\leq\dots\leq t_N=T$ be any finite partition of $[0,T]$. Then, by \eqref{eq:squarediss},
	\begin{equation}
	\sum_{i=0}^{N-1}\|V_{t_{i+1}}-V_{t_i}\|_{L^2_\eta}^2=\|W_0\|_{L^2_\eta}^2-\|V_T\|_{L^2_\eta}^2\leq\|W_0\|_{L^2_\eta}^2.
	\end{equation}
	
	As for the right continuity, the $BV^{1/2}([0,T];L^2_\eta(\Gamma;\R^d))$-property implies that for all $\bar t$ there exists the limit
	\begin{equation*}
	\lim_{t \searrow \bar t} V_t = \hat V_{\bar t},
	\end{equation*}
	and that it coincides with $V_{\bar t}$ for $\mathcal L^1$-a.e. $\bar t \in [0,T]$. Indeed, since all $\sigma$-algebras generated by $T_{t}$, $t \geq \bar t$ are contained in the one generated by $T_{\bar t}$, it follows that $\hat V_{\bar t}$ is measurable w.r.t. the $\sigma$-algebra of generated by $T_{\bar t}$. Now for every $t > \bar t$ and for every continuous function $\phi : \Gamma_{\bar t} \to \R$ the function
	\begin{equation*}
	\tilde \phi(T_{ t}(\gamma)) = \phi \big( \gamma(t) \chi_{[\bar t,t)} + T_{t}(\gamma)\chi_{[t,T]} \big)
	\end{equation*}
	is continuous and for all $\bar t \leq s \leq t$
	\begin{equation*}
	\int V_s(\gamma) \tilde \phi(T_{ t}(\gamma)) \eta(d\gamma) = \int V_t(\gamma) \tilde \phi(T_{t}(\gamma)) \eta(d\gamma),
	\end{equation*}
	because $\tilde \phi$ depends only on $T_t(\gamma)$. In particular it is constant in $[\bar t,t]$. Hence we conclude that by taking ${s} \searrow \bar t$
	\begin{equation*}
	\int \hat V_{\bar t}(\gamma) \tilde \phi(T_{t}(\gamma)) \eta(d\gamma) = \int V_t(\gamma) \tilde \phi(T_{\bar t}(\gamma)) \eta(d\gamma) = \int W_0(\gamma) \tilde \phi(T_{ t}(\gamma)) \eta(d\gamma),
	\end{equation*}
	and since as $t \searrow \bar t$ the function $\tilde \phi\circ T_{t}$ converges pointwise to $\phi\circ T_{\bar t}$, it follows that
	\begin{equation*}
	\int \hat V_{\bar t}(\gamma) \phi(T_{\bar t}(\gamma)) \eta(d\gamma) = \int W_0(\gamma) \phi(T_{\bar t}(\gamma)) \eta(d\gamma),
	\end{equation*}
	which implies $\hat V_{\bar t} = V_{\bar t}$ being $\phi$ arbitrary.
\end{proof}

\section{Approximations of dissipative solutions}\label{sec:approx}

The goal of this section is to prove Theorem \ref{thm:stickydense}, namely that dissipative solutions can be approximated by finite sticky particle solutions, which are sticky particle  solutions concentrated on a finite number of trajectories. 

We will obtain this result in four steps. In the first step we will approximate a dissipative solution by a dissipative solution whose velocity field is finitely piecewise constant in time ({called} discrete in time dissipative solution). Then we will show that discrete in time dissipative solutions can be approximated by dissipative countable particle solutions, namely dissipative solutions concentrated on a countable number of disjoint trajectories. Then we approximate dissipative countable particle solutions with dissipative finite particle solutions. Finally, we approximate dissipative finite particle solutions with finite sticky particle solutions. In each of the first three steps we will apply a general procedure which is resumed in Lemma \ref{lemma:discrete_sticky}. We start with the following definitions.  

\begin{definition}[Discrete in time dissipative solutions] 
	
	A dissipative solution $\eta\in\mathcal M(\Gamma)$ is discrete in time if there exists a finite partition $0=t_0<t_1<\dots<t_N=T$ such that $V_t=V_{t_i}$ for all $t\in[t_i,t_{i+1})$.
	
\end{definition}

\begin{definition}\label{def:countsticky}
	A dissipative countable particle solution is a discrete in time dissipative solution $\eta$ with the property that there exist{s} a countable number of trajectories $\{\gamma_n\}_{n\in\N}\subset\Gamma$ such that $\dot\gamma_n(t)=\dot{\gamma}_n(t_i)$ if $t\in[t_i,t_{i+1})$, $\inf_{n\neq m}|\gamma_n(T)-\gamma_m(T)|>0$  and $\eta(\cup_n\{\gamma_n\})=1$.
\end{definition}

\begin{definition}\label{def:finitesticky}
	A dissipative finite particle solution is a discrete in time dissipative solution $\eta$ with the property that there exist{s} a finite number of trajectories $\{\gamma_n\}_{n=1,\dots,N}\subset\Gamma$ such that $\dot\gamma_n(t)=\dot{\gamma}_n(t_i)$ if $t\in[t_i,t_{i+1})$ and  $\eta(\cup_{n=1}^N\{\gamma_n\})=1$. 
\end{definition}

In  the next lemma we exhibit a general procedure which allows to find, given a measure $\eta\in\mathcal M(\Gamma)$ (not necessarily dissipative) and a finite number of descending in time equivalence relations on $\Gamma$, a discrete in time dissipative solution.

Such a procedure  will be used several times in the following subsections to be able to approximate general  dissipative solutions with dissipative finite particle solutions.

\begin{lemma}\label{lemma:discrete_sticky}
Let $\eta\in\mathcal M(\Gamma)$, let $0=t_0<t_1<\dots<t_N=T$ be a finite partition of $[0,T]$, let $\{E_{t_i}\}_{i=0,\dots,N}$ be a family of Borel equivalence relations on $\Gamma$ with partition maps $G_{t_i}:\Gamma\to [0,1]$ such that 
\begin{equation}\label{eq:gti}
G_{t_i}(\gamma)=G_{t_i}(\gamma')\quad\Rightarrow\quad G_{t_j}(\gamma)=G_{t_j}(\gamma')\quad\forall\,j>i,\,\gamma,\gamma'\in\Gamma,
\end{equation}
or equivalently $E_{t_i}\subset E_{t_j}$ {as graphs for $t_i\leq t_j$}, and let $\{y_{T,\alpha}\}_{\alpha \in E_{T}}\subset\R^d$.
Let $\eta(d\gamma)=\int\omega^{t_i}_{y}(d\gamma){G_{t_i}}_{\#}\eta(y)$ be the disintegrations w.r.t. the partitions $\{E_{t_i}\}$, $V_{t_i}(\gamma)=\int W_0(\gamma')\omega^{t_i}_{G_{t_i}(\gamma)}(d\gamma')$ and $\bar V_t=V_{t_i}$ if $t\in[t_i,t_{i+1})$, $\bar V_t=\bar V_0$ if $t\in (-1,0)$.

Define 
\[
\tilde F:\Gamma\to\Gamma, \quad \tilde F(\gamma)=y_{T,G_T(\gamma)}-\int_t^T\bar V_s(\gamma)ds.
\]

Then, the measure $\tilde{\eta}=\tilde F_{\#}\eta$ is sticky.
\end{lemma}	

\begin{remark}\label{rem:v0}
	Observe that $W_0(\tilde F(\gamma))=\int W_0(\gamma')\omega^0_{G_0(\gamma)}(d\gamma')$ for all $\gamma\in\Gamma$. Indeed, 
	\[
	W_0(\tilde F(\gamma))=\tilde F(\gamma)(0)-\tilde F(\gamma)(-1)=\int_{-1}^0\bar V_s(\gamma)ds=\bar V_0(\gamma)=\int W_0(\gamma')\omega^0_{G_0(\gamma)}(d\gamma').
	\]
\end{remark}

\begin{proof}

	As a preliminary step, notice that, by construction, the partition induced on $\Gamma$ by $G_{t_i}$ is a refinement of the one induced by $T_{t_i}\circ \tilde F$. Indeed, if for some $\gamma'$
	\[
	\int W_0(\gamma'')\omega^{t_j}_{G_{t_j}(\gamma')}(d\gamma'')=\int W_0(\gamma'')\omega^{t_j}_{G_{t_j}(\gamma)}(d\gamma''),\quad\forall\,j\geq i,
	\]
	then $T_{t_i}\circ \tilde F(\gamma)=T_{t_i}\circ \tilde F(\gamma')$ even if possibly $G_{t_i}(\gamma)\neq G_{t_i}(\gamma')$.
	
	Therefore there exists a Borel map, say $F {: [0,1] \to \Gamma_{t_i}}$, such that $T_{t_i}\circ \tilde F=F\circ G_{t_i}$ and one has the disintegrations  
	\begin{align}
	\eta(d\gamma)&=\int\omega^{t_i}_{y}(d\gamma){G_{t_i}}_{\#}\eta(dy)\notag\\
	&=\int\omega^{t_i}_{y}(d\gamma)\nu^{t_i}_{\tilde{\gamma}}(dy){T_{t_i}}_{\#}\tilde F_{\#}\eta(d\tilde{\gamma}).
	\end{align}
	
	In particular, applying $\tilde F_{\#}$ to both sides
		\begin{align}
	\tilde F_{\#}\eta(d\tilde\gamma)
	&=\int\tilde F_{\#}\omega^{t_i}_{y}(d\tilde\gamma)\nu^{t_i}_{\tilde{\gamma}'}(dy){T_{t_i}}_{\#}\tilde F_{\#}\eta(d\tilde{\gamma}').
	\end{align}
	
	On the other hand, let $\tilde{\omega}^{t_i}_{\tilde \gamma}$ be the disintegration of $\tilde{\eta}=\tilde F_{\#}\eta$ w.r.t. the restriction map $T_{t_i}$. Then,
	\begin{equation}
		\tilde F_{\#}\eta(d\tilde\gamma)
	=\int\tilde\omega^{t_i}_{\tilde{\gamma}'}(d\tilde\gamma){T_{t_i}}_{\#}\tilde F_{\#}\eta(d\tilde{\gamma}').
	\end{equation}
	
	Hence,
	
	\begin{equation}\label{eq:rhopush}
	\tilde \omega^{t_i}_{T_{t_i}\circ\tilde F(\gamma)}(d\tilde{\gamma}')=\int\tilde F_{\#}\omega^{t_i}_{y}(d\tilde\gamma')\nu^{t_i}_{T_{t_i}\circ \tilde F(\gamma)}(dy).
	\end{equation}

Let us now check the dissipation condition for $\tilde \eta$ first when $t=t_i$ for some $i\in\{0,\dots,N\}$, considering here the right derivative because of Corollary \ref{cor:BV}.
Recall that the vector field associated to $\tilde\eta$ is given by
\[
\tilde V_t(\tilde{\gamma})=\int W_0(\tilde{\gamma}')\tilde{\omega}^{t_i}_{T_{t_i}(\tilde{\gamma})}(d\tilde{\gamma}'), \qquad {t \in [t_i,t_{i+1})}.
\]

One has that 

\begin{align}
{\dot{(\widetilde{F}(\gamma))}}(t_i)&=\bar V_{t_i}(\gamma)=\int W_0({\gamma}')\omega^{t_i}_{ G_{t_i}(\gamma)}(d\gamma')\notag\\
&=\int W_{0}(\gamma')\omega^{t_i}_{y}(d\gamma')\nu^{t_i}_{T_{t_i}\circ \tilde F(\gamma)}(dy),
\end{align}
where in the last equality we have used the fact $\int W_{0}(\gamma')\omega^{t_i}_{(\cdot)}(d\gamma')$ is constant on $F^{-1}(T_{t_i}\circ\tilde F(\gamma))$.

Finally, by Remark \ref{rem:v0} and \eqref{eq:rhopush}
\begin{align}
{\dot{(\widetilde{F}({\gamma}))}}(t_i)&=\int W_{0}(\gamma')\omega^{t_i}_{y}(d\gamma')\nu^{t_i}_{T_{t_i}\circ \tilde F(\gamma)}(dy)\notag\\
&=\int W_0(\tilde{\gamma}')\tilde F_{\#}\omega^{t_i}_{y}(d\gamma')\nu^{t_i}_{T_{t_i}\circ \tilde F(\gamma)}(dy)\notag\\
&=\int W_0(\tilde{\gamma}')\tilde \omega^{t_i}_{T_{t_i}\circ\tilde F(\gamma)}(d\tilde{\gamma}')\notag\\
&=\tilde V_{t_i}(\tilde F(\gamma)).
\end{align} 

Let us now assume $t\in[t_i,t_{i+1})$. Then, 
\[
{\dot{(\widetilde F(\gamma))}}(t)=\bar V_t(\gamma)=\bar V_{t_i}(\gamma)=\tilde V_{t_i}(\tilde F(\gamma)).
\]

Let us now prove that for $t\in[t_i,t_{i+1})$ $\tilde V_t(\tilde F(\gamma))=\tilde V_{t_i}(\tilde F(\gamma))$.

This follows from the following observation
\begin{equation}\label{eq:tti}
T_t\circ\tilde F(\gamma)=T_t\circ\tilde F(\gamma')\quad\Rightarrow\quad T_{t_i}\circ \tilde F(\gamma)=T_{t_i}\circ \tilde F(\gamma'),
\end{equation}
which can be easily checked from the definition of $\tilde F$. 

From \eqref{eq:tti} one has that $\tilde \omega^t_{T_t\circ\tilde F(\gamma)}(d\tilde{\gamma})=\tilde \omega^{t_i}_{T_{t_i}\circ\tilde F(\gamma)}(d\tilde{\gamma})$, hence $\tilde V_t(\tilde F(\gamma))=\tilde V_{t_i}(\tilde F(\gamma))$.
\end{proof}

\subsection{Discrete in time  approximation}

We now apply the general construction of Lemma \ref{lemma:discrete_sticky} in order to prove that, given a dissipative solution $\eta$, there exists a discrete in time dissipative solution  whose vector field is close to the vector field of $\eta$ at any time in the $L^2_{\eta}$ topology.

\begin{proposition}\label{prop:discretesticky}
Let $\eta\in\mathcal M(\Gamma)$ be a sticky particle solution and $\eps>0$. Then there exists $\tilde \eta^\eps=\tilde F^\eps_{\#}\eta\in\mathcal M(\Gamma)$ discrete in time sticky particle solution such that
\begin{equation}\label{eq:etaepsprop}
W_0(\tilde F^\eps(\gamma))=W_0(\gamma),\quad\|\tilde V^\eps_t\circ \tilde F^\eps-V_t\|_{L^2_\eta}^2\leq \eps^2, \quad\int|\tilde F^\eps(\gamma)(t)-\gamma(t)|^2\eta(d\gamma)\leq {\eps^2} T.
\end{equation}

In particular, as $\eps\to0$, the measures $\eta^\eps$ converge in $\mathcal M(\Gamma)$ to $\eta$.	

	\end{proposition}

\begin{proof}
	By Corollary \ref{cor:BV}, we {k}now that the velocity field {$t\mapsto V_t$} of $\eta$ belongs to $BV^{1/2}([0,T];L^2_\eta(\Gamma;\R^d))$ and it is right continuous. Therefore, given $\eps>0$, there exists a finite partition $0=t_0<t_1<\dots<t_N=T$ such that either
	\begin{align}\label{eq:part1}
	\|V_{t_i}\|_{L^2_\eta}^2-\|V_{t_{i+1}}\|_{L^2_\eta}^2<\eps^2
	\end{align}
	or \begin{align}\label{eq:part2}
	\|V_{t_i}\|_{L^2_\eta}^2-\lim _{s\nearrow t_{i+1}}\|V_s\|_{L^2_\eta}^2<\eps^2,\qquad\lim_{s\nearrow t_{i+1}}\|V_s\|_{L^2_\eta}^2-\|V_{t_{i+1}}\|_{L^2_\eta}^2\geq\eps^2/2.
	\end{align}
	Moreover, there are at most $2\|W_0\|_{L^2_\eta}^2/\eps^2$ points where the last situation occurs.
	
	We apply now Lemma \ref{lemma:discrete_sticky} to $\eta$, $0=t_0<t_1<\dots<t_N=T$, $G_{t_i}=T_{t_i}$, $y_{T,G_T(\gamma)}=\gamma(T)$ and find $\tilde \eta^\eps=\tilde F^\eps_{\#}\eta$ discrete in time dissipative solution. 
	
	In particular, since $G_{t_i}=T_{t_i}$, then the first property in \eqref{eq:etaepsprop} holds.
	
	Let us now check the {remaining} properties in \eqref{eq:etaepsprop}. Let $t\in[t_i,t_{i+1})$: by \eqref{eq:part1}, \eqref{eq:part2}
	
	\begin{align}
	\|\tilde V^\eps_t\circ \tilde F^\eps-V_t\|_{L^2_\eta}^2=\|V_{t_i}-V_t\|_{L^2_\eta}^2=\|V_{t_i}\|_{L^2_\eta}^2-\|V_t\|_{L^2_\eta}^2<\eps^2.
	\end{align}
	Moreover, using \eqref{eq:part1} and \eqref{eq:part2} and Jensen inequality,
	\begin{align}
	\int|\tilde F^\eps(\gamma)(t)-\gamma(t)|^2\eta(d\gamma)\leq \int(T-t) {\biggl(} \int_t^T|\bar V_s(\gamma)-V_s(\gamma)|^2ds {\biggr)}\eta(d\gamma)\leq\eps^2T. 
	\end{align}
\end{proof}

\subsection{Dissipative countable particle approximation}

Given a discrete in time dissipative solution, we now construct a discrete in time dissipative approximation with the additional property of being concentrated on a countable set of trajectories.

\begin{proposition}\label{prop:countsticky}
	Let $\eta$ be a discrete in time dissipative solution and $\delta>0$. Then there exists $\hat{\eta}^\delta=\hat F^\delta_{\#}\eta$ discrete in time  dissipative  solution such that  
	
		\begin{equation}\label{eq:etadeltaprop}
	|V_t(\hat F^\delta(\gamma))-V_t(\gamma)|\leq\delta, \quad|\hat F^\delta(\gamma)(t)-\gamma(t)|\leq \delta (1+T).
	\end{equation}
	Moreover, $\hat{\eta}^\delta$ is a  dissipative countable  particle solution in the sense of Definition \ref{def:countsticky}.
\end{proposition} 

\begin{proof}
	Let $\{Q_k\}_{k\in\N}$ be a countable partition of $\R^d$ into cubes of side length $\delta$, $Q_k=x_k+[-\delta/2,\delta/2)^d$. 
	If $0=t_0<t_1<\dots<t_N=T$ is the partition associated to $\eta$, define $\forall\,i=0,\dots,N$ a partition map $G_{t_i}$ on $\Gamma$ defining the following equivalence relation{:}
	\begin{equation}
	G_{t_i}(\gamma)=G_{t_i}(\gamma')\quad\Leftrightarrow\quad \forall\,j\geq i,\,\exists\,k_j: V_{t_j}(\gamma)\in Q_{k_j},\,V_{t_j}(\gamma')\in Q_{k_j}\text{ and }\exists \bar k_T: \gamma(T)\in Q_{\bar k_T}, \gamma'(T)\in Q_{\bar k_T}.
	\end{equation}
	In particular, the maps $G_{t_i}$ satisfy condition \eqref{eq:gti}. 
	Let moreover $y_{T,G_T(\gamma)}=x_{\bar k_T}$ if $\gamma(T)\in Q_{\bar k_T}$, observing that $\inf_{n\neq m}|\gamma_n(T)-\gamma_m(T)|\geq\delta>0$.
	Denote by $\bar V^\delta_t$ the discrete in time vector field associated with such partitions. 
	
	Apply now Lemma \ref{lemma:discrete_sticky} finding a discrete in time  dissipative  solution $\hat{\eta}^\delta=\hat F^\delta{\#}\eta$.
	
	By construction, $\hat{\eta}^\delta$ is a  dissipative countable particle solution and it satisfies \eqref{eq:etadeltaprop}.
\end{proof}

\subsection{Dissipative finite particle approximation}

Now we want to approximate  dissipative countable particle solutions with  dissipative finite particle solutions.

\begin{proposition}\label{prop:finitesticky}
	Let $\eta\in\mathcal M(\Gamma)$ be a  dissipative countable particle solution. Then, for every $\sigma>0$, there exists $\eta^\sigma$  dissipative finite particle solution with the property that $\eta^\sigma\rightharpoonup\eta$ as $\sigma\to0$.
	\end{proposition}

\begin{proof}
Since $\eta\in\mathcal M(\Gamma)$ is a dissipative  solution, 
\[
\int|\gamma(T)|^2\eta(d\gamma)\leq\int|\gamma(0)|^2\eta(d\gamma)+T^2\int|W_0(\gamma)|^2\eta(d\gamma)\leq 1+T^2
\]
Hence, for any $\Lambda>0$
\begin{equation}
\eta(\{\gamma:\,|\gamma(T)|>\Lambda\})\leq\frac{1+T^2}{\Lambda^2}.
\end{equation}

As a first step, take then $\bar{\eta}^{\sigma}=\eta_{\llcorner\{\gamma:\,|\gamma(T)|\leq\Lambda\}}/\eta(\{\gamma:|\gamma(T)|\leq\Lambda\})$ for $\Lambda$ large enough so that $\|\bar{\eta}^\sigma-\eta\|_{\mathcal M(\Gamma)}\leq\sigma/3$.

Then we want to apply Lemma \ref{lemma:discrete_sticky} to $\bar{\eta}^\sigma$ with the partition given by the maps $T_{t_i}$,  the finite end points $y_{T, T_T(\gamma)}=\gamma(T)$ and vector field $\bar V^{\Lambda}_t$ defined as usual starting from a modification $W_0^\Lambda$ of $W_0$ defined as follows{:}
\begin{equation}
W_0^\Lambda(\gamma_n)=W_0(\gamma_n) \text{ if $|W_0(\gamma_n)|\leq\Lambda$ },\quad W_0^\Lambda(\gamma_n)=\frac{W_0(\gamma_n)}{|W_0(\gamma_n)|}\Lambda\text{ if $|W_0(\gamma_n)|>\Lambda$ }. 
\end{equation}

In this way one obtains a countable sticky particle solution $\hat{\eta}^\sigma$ which,  if $\Lambda$ is large enough, satisfies $\hat V^\sigma_t\circ\hat F^\sigma(\gamma)=\bar V_t^\Lambda(\gamma)$, and is close to $\bar{\eta}^{\sigma}$ in the weak topology.

Moreover, since now the velocity fields $\hat V^\sigma_{t_i}(\hat F^{\sigma}(\gamma_n))$ are all contained in the ball of radius $\Lambda$ in $\R^d$, the same construction performed  in Proposition \ref{prop:countsticky} with $\delta\leq\sigma/3$ leads to a finite sticky particle solution $\eta^\sigma$ with the desired properties.
\end{proof}

\subsection{Finite sticky particle solutions}

Finally we can prove the density of finite sticky particle solutions in the set of dissipative solutions.

\begin{theorem}\label{thm:stickydense}
	Let $\eta\in\mathcal M(\Gamma)$ be a dissipative solution. {T}hen there {exist} $\eta^\nu$ sticky particle solutions with finitely many trajectories (finite sticky particle solutions) with the property that $\eta^\nu\rightharpoonup\eta$ as $\nu\to 0$. 
	\end{theorem}

\begin{proof}
	By Propositions \ref{prop:discretesticky}, \ref{prop:countsticky} and \ref{prop:finitesticky} we can assume that $\eta\in\mathcal M(\Gamma)$ is a dissipative finite particle solution. 
	Then we proceed as in the construction of Lemma \ref{lemma:discrete_sticky}, where at each step $[t_i,t_{i+1})$ we perturb the speed $V_{t_i}$ into a vector field $\tilde V_{t_i}$ in order to have that the trajectories $\gamma_n(t_{i+1}) - \tilde V_{t_i}(\gamma_n) t$ do not intersect in $[t_i,t_{i+1})$. Being the intersection conditions a set of codimension $d-1$, it is fairly easy to see that we can assume $\tilde V_{t_i}$ arbitrarily close to $V_{t_i}$.
\end{proof}

\section{PDE formulations}
\label{S:PDE_formul}

In this section we give a kinetic and PDE formulation of our notion of solution.

Define the kinetic measure $\varpi_t \in \mathcal P(\R^d \times \R^d)$
\begin{equation*}
\int \phi(x,v) \varpi_t(dxdv) = \int \phi(\gamma(t),V_t(\gamma)) \eta(d\gamma).
\end{equation*}

\begin{proposition}
	\label{prop:kinetic_equation}
	The measure $\varpi_t$ satisfies the PDE
	\begin{equation}
	\label{eq:kinetic_varpi}
	\partial_t \varpi_t + v \cdot \nabla_x \varpi_t + \mathrm{div}_v \pi = 0,
	\end{equation}
	where $\pi$ is a distribution such that
	\begin{equation*}
	\langle\phi, \pi\rangle \leq \|\nabla_v \phi\|_{C^0}
	\end{equation*}
	and for every function $\Psi(t,x,v)$ convex in $v$ with $\Psi(t,x,v) \leq C \psi(t,x) (1+|v|^2)$ for some $C > 0$, $\psi \in C^1_c(\R^d \times \R^+)$ it holds
	\begin{equation}
	\label{eq:dissip_varpi}
	\langle\Psi, \mathrm{div}(\pi)\rangle \geq 0.
	\end{equation}
\end{proposition}

The requirement that $\Psi$ has quadratic growth and compact support in $(t,x)$ implies that it can be used as a test function for \eqref{eq:kinetic_varpi}.

\begin{proof}
	Since the weak formulation is invariant for weak limits, we can write the PDE  \eqref{eq:kinetic_varpi} for the approximate finite sticky particle solutions  found in Theorem  \ref{thm:stickydense}, and pass to the limit in the estimates obtained. For finite sticky particle solutions 
	\begin{equation*}
	\varpi_t = \sum_n c_n \delta_{(x_n(t),v_n(t))},
	\end{equation*}
	so that
	\begin{equation*}
	\begin{split}
	\partial_t \varpi_t + v \cdot \nabla_x \varpi_t &= \sum_{i} \sum_n c_n \big[ \delta_{(x_n(t_i),v_n(t_i))} - \delta_{(x_n(t_i),v_n(t_{i-1}))} \big] \\
	&= \sum_i \sum_{x_{ij}} \sum_{x_n(t_i) = x_{ij}} c_n \big[ \delta_{(x_{ij},v_n(t_i))} - \delta_{(x_{ij},v_n(t_{i-1}))} \big],
	\end{split}
	\end{equation*}
	where in the last equality we have used the fact that the variation of speed occurs only at the times $t_i$ in finite many points $x_{ij}$. We can write the r.h.s. also in divergence form as
	\begin{equation*}
	\partial_t \varpi_t + v \cdot \nabla_x \varpi_t = \sum_i \sum_{x_{ij}} \sum_{x_n(t_i) = x_{ij}} c_n \mathrm{div}_v \bigg( \frac{v_n(t_{i-1}) - v_n(t_i)}{|v_n(t_{i-1}) - v_n(t_i)|} \mathcal H^1 \llcorner_{\{(1-\ell) v_n(t_{i}) + \ell v_n(t_{i-1}),\ell \in [0,1]\}} \bigg),
	\end{equation*}
	which shows that the equation for $\varpi_t$ is in divergence form as in \eqref{eq:kinetic_varpi}.
	
	We now show that $\pi$ is a distribution which can be computed on $C^1$-functions (i.e. it is a first order distribution). Recall that for a $C^1$ function $\bar \phi$
	\begin{equation*}
	\begin{split}
	\bar\phi(v) - \bar \phi(\bar v) &= \bigg( \int_0^1 \nabla \bar \phi \big( (1-l) v_1 + l v_2 \big) dl \bigg) \cdot (v_2 - v_1),
	\end{split}
	\end{equation*}
	so that testing
	\begin{equation*}
	- \pi = \sum_i \sum_{x_{ij}} \sum_{x_n(t_i) = x_{ij}} c_n \bigg( \frac{v_n(t_{i-1}) - v_n(t_i)}{|v_n(t_{i-1}) - v_n(t_i)|} \mathcal H^1 \llcorner_{\{(1-\ell) v_n(t_{i}) + \ell v_n(t_{i-1}),\ell \in [0,1]\}} \bigg)
	\end{equation*}
	with a $C^1$ function $\phi=\phi(t,x,v)$ and using
	\begin{equation}
	\label{eq:distrte_classic_aver_sped}
	v_n(t_i) = \frac{\sum_{x_n(t_i) = x_{ij}} c_n v_n(t_{i-1})}{\sum_{x_n(t_i) = x_{ij}} c_n}.
	\end{equation}
	one obtains
	\begin{equation*}
	\begin{split}
	- \int \phi \pi &= \sum_i \sum_{x_{ij}} \sum_{x_n(t_i) = x_{ij}} c_n \big( v_n(t_{i-1}) - v_n(t_i) \big) \int_0^1 \phi \big( t_i,x_n(t_i),(1 -\ell) v_n(t_i) + \ell v_n(t_{i-1}) \big) d\ell \\
	&= \sum_i \sum_{x_{ij}} \sum_{x_n(t_i) = x_{ij}} c_n \big( v_n(t_{i-1}) - v_n(t_i) \big) \phi(t_i,x_n(t_i),v_n(t_i)) \\
	& \quad + \sum_i \sum_{x_{ij}} \sum_{x_n(t_i) = x_{ij}} c_n \big( v_n(t_{i-1}) - v_n(t_i) \big) \\
	&{ \qquad \qquad \qquad \qquad \bigg[ \bigg( \int_0^1 (1 - \ell) \nabla_v \phi \big( t_i,x_n(t_i),(1 - \ell) v_n(t_i) + \ell v_n(t_{i-1}) \big) d\ell \bigg) \cdot \big( v_n(t_{i-1}) - v_n(t_i) \big) \bigg] }\\
	&= \sum_i \sum_{x_{ij}} \sum_{x_n(t_i) = x_{ij}} c_n \big( v_n(t_{i-1}) - v_n(t_i) \big) \\
	& \qquad \qquad \qquad \qquad \bigg[ \bigg( \int_0^1 (1 - \ell) \nabla_v \phi \big( t_i,x_n(t_i),(1 - \ell) v_n(t_i) + \ell v_n(t_{i-1}) \big) d\ell \bigg) \cdot \big( v_n(t_{i-1}) - v_n(t_i) \big) \bigg].
	\end{split}
	\end{equation*}
	Using the dissipation of energy proved in Proposition \ref{prop:diss} it follows that
	\begin{equation*}
	{\big|} \langle  \phi, \pi\rangle {\big|} \leq \|\nabla_v \phi\|_{C^0} \|W_0\|_{L^2_\eta}^2.
	\end{equation*}
	This proves the first claim about $\pi$, since $\eta \in \mathcal M(\Gamma)$.
	
	Testing $\div_v\pi$ with a function $\Psi(t,x,v)$ convex w.r.t. $v$ we obtain
	\begin{equation*}
	\begin{split}
	\langle\Psi,\div_v\pi\rangle=-
	\sum_i \sum_{x_{ij}} \sum_{x_n(t_i) = x_{ij}} c_n \big[ \Psi(t_i,x_n(t_i),v_n(t_i)) - \Psi(t_{i-1}, x_n(t_{i-1}), v_n(t_{i-1})) \big] \geq 0
	\end{split}
	\end{equation*}
	by Jensen inequality  and \eqref{eq:distrte_classic_aver_sped}. 
	This concludes the proof.
\end{proof}

Now that we have a kinetic formulation one can give the following Eulerian  formulation. Define 
\begin{equation*}
\varpi_t(dxdv) = \int \varpi_{t,x}(dv) \rho_t(dx), \quad \rho_t = (\mathtt p_x)_\sharp \varpi_t,
\end{equation*}
and
\begin{equation*}
u_t(x) = \int v \varpi_{t,x}(dv), \quad w_t(x) = \int (v \otimes v) \varpi_{t,x}(dv),
\end{equation*}
we can observe that \eqref{eq:dissip_varpi} implies that
\begin{equation}
\label{Equa:Euler_equation}
\partial_t \rho_t + \mathrm{div}_x (\rho_t u_t) = 0, \quad \partial_t (\rho_t u_t) + \mathrm{div}_x (\rho_t w_t) = 0{,}
\end{equation}
which is the Eulerian formulation of the sticky particle system. Using instead $\phi(t,x) |v|^2$ we deduce that
\begin{equation*}
 \phi(t,x)\langle  |v|^2, \mathrm{div}_v \pi\rangle = \mu \in \mathcal M^+(\R^+\times \R^d),
\end{equation*}
which embodies the dissipation of energy. However in general we cannot compute the third order moment {$\int |v|^2v \varpi_{t,x}(dv)$} as a Lebesgue integral, so the energy balance cannot be written in terms of $\rho_t$ integrable functions.

\section{A $G_\delta$ dense set of  initial data}
	\label{sec:gdelta}
	
	In this section we prove that there is a $G_\delta$ dense set of initial data for which any sticky particle solution departing from them is a flow in which particles do not interact. In order to make this statement precise, we need to introduce the following definitions.

\begin{definition}[Free flow]
	We say that a dissipative solution $\eta\in\mathcal M(\Gamma)$ is a free flow if $\eta$ is concentrated on a set of straight lines with empty mutual intersection. 
\end{definition}


\begin{definition}[Initial data for sticky particle solutions]
	Let $W_0:\Gamma\to\R^d$ be the continuous map defined by $W_0(\gamma)=\gamma(0)-\gamma(-1)$. We say that a probability measure $\nu_0\in\mathcal P(\R^d\times\R^d)$ is an initial data of a dissipative solution if there exists $\eta\in\mathcal M(\Gamma)$ dissipative solution s.t. $\nu_0=(e_0,W_0)_{\#}\eta$, i.e.
	\[
	\int\phi(x,v)\nu_0(dxdv)=\int\phi(\gamma(0),W_0(\gamma))\eta(d\gamma).
	\]
\end{definition}

Notice that, by compactness of the set of dissipative  solutions and by continuity of the map $W_0$, the set of initial data of dissipative solutions is compact as well. Moreover, since finite convex combinations of Dirac deltas on $\R^d$ pointing in different directions give always rise to a sticky particle solutions, by density it actually coincides with
\[
P_{2,1}(\R^d\times\R^d)=\Bigl\{\nu_0\in\mathcal P(\R^d\times\R^d):\,\int|x|^2{\mathtt{p}_x}_{\#}\nu_0\leq1,\,\int|v|^2{\mathtt{p}_v}_{\#}\nu_0\leq1\Bigr\}.
\]

The main result of this section is Theorem \ref{thm:freeflow0}, which we recall below:

\begin{theorem}\label{thm:freeflow}
	There is a set $D_0\subset\mathcal P_{2,1}(\R^d\times\R^d)$ such that, for any $\nu_0\in D_0$ there exists a unique dissipative  solution $\eta$ with initial data $\nu_0$ and it is given by a free flow. Such a set is a dense $G_{\delta}$ set (i.e. of second category)  in the weak topology on $\mathcal P_{2,1}(\R^d\times\R^d)$. 
\end{theorem}

Being the map $e_0\times W_0$ continuous, also $(e_0\times W_0)^{-1}(D_0)$ is a $G_\delta$-set. But it cannot be dense in the set of dissipative solutions: just consider a finite sticky particle solution where trajectories do interact. 

In order to prove Theorem \ref{thm:freeflow} we need the following preliminary lemmas.

\begin{lemma}\label{lemma:freeflowfinite}
	If
	\begin{equation*}
	\nu_0(dxdv) = \sum_{n=1}^N c_n \delta_{(x_n,v_n)}(dxdv)
	\end{equation*}
	and the trajectories $x_n + \R^+ v_n$ do not intersect, then there is a unique  $\eta\in\mathcal M(\Gamma)$ dissipative solution s.t. $\nu_0=(e_0,W_0)_{\#}\eta$  given by 
	\begin{equation*}
	\eta(d\gamma) = \sum_{n=1}^N c_n \delta_{\{x_n + t v_n\}}(d\gamma).
	\end{equation*}
\end{lemma}

\begin{lemma}\label{lemma:straight}
	Let $\eta\in\mathcal M(\Gamma)$ be such that $D(\eta)=0$. Then $\eta$ is concentrated on straight lines of the form $\gamma(0)+tW_0(\gamma)$.
\end{lemma}

Define 

	\begin{equation}\label{eq:D_0}
D_0 = \Big\{ \nu_0\in\mathcal P_{2,1}(\R^d\times\R^d):\, D(\eta)=0\quad\forall\,\eta\in\mathcal M(\Gamma)\text{ dissipative  solutions s.t. }(e_0,W_0)_{\#}\eta=\nu_0 {\Big\}}.
\end{equation}

\begin{lemma}\label{lemma:nodiss}
	If $\nu_0\in D_0$, then $\nu_0$ is concentrated on a graph $(x_0,v_0(x_0))\subset\R^d\times\R^d$ of a map $v_0$ such that the straight lines $[0,T]\ni t\mapsto x_0+tv_0$ do not intersect and there is a unique $\eta\in\mathcal M(\Gamma)$ dissipative solution with $(e_0,W_0)_{\#}\eta=\nu_0$ given by the free flow concentrated on these straight lines.
\end{lemma}

\begin{proof}[Proof of Theorem \ref{thm:freeflow}]
	
	Thanks to Lemma \ref{lemma:nodiss}, it is sufficient to prove that the set $D_0$ defined in \eqref{eq:D_0} is a dense  $G_\delta$ set  in the weak topology.

	Let us first prove that $D_0$ is dense in $\mathcal P_{2,1}(\R^d\times\R^d)$. 
	
	Consider then any initial data $\nu_0=(e_0,W_0)_{\#}\eta\in\mathcal P_{2,1}(\R^d\times\R^d)$. Applying first Proposition \ref{prop:discretesticky}, then Proposition \ref{prop:countsticky} and finally Proposition \ref{prop:finitesticky} we find a sequence of finite sticky particle solutions $\eta^\sigma$ weakly converging to $\eta$ in $\mathcal M(\Gamma)$. By continuity of $W_0$, the measures $\nu_0^\sigma=(e_0,W_0)_{\#}\eta^\sigma$ converge weakly to $\nu_0$. 
	
	Since $\eta^\sigma$ is a {sticky} finite particle solution and the dimension $d$ is greater or equal than $2$, it is not difficult to see that it is possible to modify slightly the initial datum $\nu_0^\sigma$ to obtain an initial datum $\bar\nu_0^\sigma=(e_0,W_0)_{\#}\bar\eta^\sigma$ in $D_0$. Indeed, let

	\begin{equation*}
	\nu_0^\sigma = \sum_n  \eta^\sigma(\{\gamma_n\}) \delta_{(\gamma_n(0),W_0(\gamma_n))}(dxdv), 
	\end{equation*}
	be the initial datum of a finite sticky particle solution.
     The set of initial speeds $\tilde W_0$ for which
	\begin{equation*}
	\exists n \not= m \text{ s.t. } \dist \Big( \big\{\gamma_n(0) + t W_0(\gamma_n),\, t\in [-1,T]\big\}, \big\{\gamma_m(0) + t W_0(\gamma_m),\, t\in[-1,T]\big\} \Big) = 0
	\end{equation*}
	is {contained in} the set 
	\begin{equation*}
	\exists n \not= m \text{ s.t. } \gamma_n(0) - \gamma_m(0) \parallel W_0(\gamma_n) - W_0(\gamma_m),
	\end{equation*}
	which is closed and has codimension $d-1$. Therefore any initial datum $\nu_0$ of a finite sticky particle solution with $W_0$ belonging to the open and dense complement of the above set has a free flow solution with a finite number of trajectories at a strictly positive mutual distance. Now we apply Lemma \ref{lemma:freeflowfinite} to obtain that $\bar{\nu}_0^\sigma$ generates a unique dissipative solution 
	given by a free flow, namely $\bar{\nu}_0^\sigma\in D_0$.

	Having proved the density of $D_0$, let us prove it is given by the intersection of countably many open sets.
	
	Define, for all $k\in\N$, the sets
	\begin{equation}
	D_{1/k} = {\Big\{} \nu_0\in\mathcal P_{2,1}(\R^d\times\R^d):\, D(\eta)<1/k\quad\forall\,\eta\in\mathcal M(\Gamma)\text{ dissipative solutions s.t. }{(e_0,W_0)}_{\#}\eta=\nu_0 {\Big\}}.
	\end{equation}
	
	It is clear that 
	\[
	D_0=\bigcap_{k}D_{1/k}.
	\]
	
	We claim that the sets $D_{1/k}$ are open sets. Indeed, if this were not the case, there would be $\nu_0\in D_{1/k}$ and a sequence of initial data $\nu_0^n\rightharpoonup\nu_0$ each generating a dissipative solution  $\eta^n$ such that $D(\eta^n)\geq1/k$. By compactness of the set of dissipative solutions, up to extracting a subsequence and relabelling it the measures $\eta^n$ converge to a dissipative solution $\eta$ with $(e_0,W_0)_{\#}\eta=\nu_0$. By definition, it is not difficult to check that the total dissipation is upper semicontinuous w.r.t. weak converge of measures: indeed the measure 
	\[
	\int\phi(t,x,v)\varpi_n(dtdxdv)=\int{\biggl[}\int\phi(t,\gamma(t),\dot{\gamma}(t)){\biggr]}\eta^n(d\gamma)
	\]
	converges weakly for all $\phi$ continuous, i.e. $W_0$, and being $|v|^2$ convex
	\[
	\int\int_0^T|V_t(\gamma)|^2dt\eta(d\gamma)=\int|v|^2{\varpi}\leq\liminf_n\int|v|^2\varpi_n=\liminf_n\int\int_0^T|V_t(\gamma)|^2dt\eta(d\gamma).
	\]
	 Therefore one should have that $D(\eta)\geq1/k$, contrary to the assumption on the initial data $\nu_0$.    
\end{proof}

Let us now prove the series of preliminary lemmas.

\begin{proof}[Proof of Lemma \ref{lemma:freeflowfinite}]
		First of all, by Proposition \ref{prop:diss}, it follows that $\eta$ is concentrated on the set of trajectories $\gamma$ such that
		\begin{equation*}
		\|\dot \gamma\|_\infty \leq \max_n |v_n| = \bar V.
		\end{equation*}
		Hence if $\bar d$ is the minimal distance among the sets $\{x_n + \R^+ v_n\}$, a trajectory $\gamma$ starting in $x_n$ needs a time of order $\bar t = \bar d/2\bar V$ before interacting with a trajectory $\gamma'$ starting in $x_m \not= x_n$. In the interval of time $[0,\bar t]$ the partition $\gamma \mapsto \gamma(0)$ is then the least sharp partition allowed, and being the initial one one concludes that $\dot \gamma(t) = \dot \gamma(0)$. The statement is obtained by repeating the argument for every interval of time $[n,n+1] \bar t$.
\end{proof}

\begin{proof}[Proof of Lemma \ref{lemma:straight}:]
If $D(\eta)=0$, then the map $(-1,T]\ni t\mapsto\int|V_t(\gamma)|^2\eta(d\gamma)$ is constant, where we have used the usual convention $V_t=W_0$ if $t\in(-1,0)$.
Recalling the definition of $V_t$, this implies that
\[
\int|W_0(\gamma')|^2\omega_{T_t(\gamma)}(d\gamma')= {\biggl|} \int W_0(\gamma')\omega_{T_t(\gamma)}(d\gamma') {\biggr|}^2
\]	
that in turn by Jensen inequality implies that $W_0$ is constant on $T_t^{-1}(T_t(\gamma))$ for all $\gamma$ and for all $t$. Hence,
\[
\dot{\gamma}(t)=V_t(\gamma)=W_0(\gamma) \quad\mathcal L^1\times \eta\text{{-}a.e. in $(-1,T]\times\Gamma$},
\] 	
namely $\eta$ is concentrated on straight lines with velocity $W_0(\gamma)$.
\end{proof}

\begin{proof}[Proof of Lemma \ref{lemma:nodiss}:]
	
	Let $\eta\in\mathcal M(\Gamma)$ dissipative solution with $(e_,W_0)_{\#}\eta=\nu_0\in D_0$. By Lemma \ref{lemma:straight} we know that $\eta$ is concentrated on a set $\Delta\subset\Gamma$ of straight lines.
	
	From now on, we identify straight lines  $\gamma(t)=x_0+v_0t\in\Gamma$ with points $(x_0,v_0)\in\R^d\times\R^d$ according to our convenience.
	
	Define the set
	\begin{equation*}
	H = \Big\{ \big( x_0,v_0,x_0',v_0' \big) \in \R^{2d} \times \R^{2d} : (x_0,v_0) \not= (x_0',v_0') \ \text{and} \ \exists t \in [0,T] \text{ s.t. } x_0 + v_0 t = x_0' + v_0' t \Big\}.
	\end{equation*}
	
	We will use also the notation
	\begin{equation*}
	\gamma(t) = x_0 + v_0 t, \quad \gamma'(t) = x_0' + v_0' t,
	\end{equation*}
	and in some cases consider $H$ as a subset of $\Gamma \times \Gamma$. Let $t_{\gamma,\gamma'}$ be the first crossing time in the definition of $H$.
	
	For every $(x_0,v_0,x_0',v_0') \in H$ consider the map
	\begin{equation}
	\label{Equa:gamma_gamma_prime_map}
	\left. \begin{array}{c}
	\gamma(t) = x_0 + v_0 t, \\
	\gamma'(t) = x_0' + v_0' t,
	\end{array} \right\}
	\quad \mapsto \quad
	\begin{cases}
	\tilde \gamma_{\gamma'}(t) = \gamma(t) \chi_{t \leq t_{\gamma,\gamma'}} + \frac{\gamma(t) + \gamma'(t)}{2} \chi_{t > t_{\gamma,\gamma'}}, \\
	\tilde \gamma'_{\gamma}(t) = \gamma'(t) \chi_{t \leq t_{\gamma,\gamma'}} + \frac{\gamma(t) + \gamma'(t)}{2} \chi_{t > t_{\gamma,\gamma'}}.
	\end{cases}
	\end{equation}
	The effect of the map is to replace the curves $\gamma,\gamma'$ at the first crossing point with the line with their average speed.
	
	
	We will use the duality results of \cite[Proposition 3.3]{kel:duality}: if $B$ is Borel (or analytic)
	\begin{equation}
	\label{Equa:kellerer_1}
	\sup_{\pi \in \mathrm{adm}(\mu_1,\mu_2)} \pi(B) = \inf \bigg\{ \mu_1(B_1) + \mu_2(B_2):\, B_1 \times X \cup X \times B_2 \supset B \bigg\}.
	\end{equation}
	We have used the notation
	\begin{equation*}
	\mathrm{adm}(\mu_1,\mu_2) = \Big\{ \pi \in \mathcal M^+(X \times Y):\, (\mathtt p_x)_\sharp \pi \leq \mu_1, (\mathtt p_y)_\sharp \pi \leq \nu_2 \Big\}.
	\end{equation*}
	
	Consider an admissible transference plan $\pi$ concentrated in $H\cap \Delta\times\Delta$: the plan is said to be admissible if
		\begin{equation*}
		(\mathtt p_{x_0,v_0})_\sharp \pi \leq \nu_0, \quad (\mathtt p_{x_0',v_0'})_\sharp \pi \leq \nu_0.
		\end{equation*}
		Being $H$ symmetric, we can assume that $\pi$ is also symmetric, i.e. $\pi(A \times B) = \pi(B \times A)$. Define the disintegration
		\begin{align*}
		\pi(d\gamma d\gamma')& = \int \pi_\gamma(d\gamma')\eta_{\llcorner\mathtt p_{x_0,v_0}(H\cap \Delta\times\Delta)}(d\gamma) \notag\\
		&=\int \pi_\gamma(d\gamma')\eta(d\gamma),
		\end{align*}
		where we set $\pi_{\gamma}=0$ if $\gamma\notin\mathtt p_{x_0,v_0}(H\cap \Delta\times\Delta)$.

		Define the new Lagrangian representation $\tilde \eta$ as follows: 
		\begin{equation}
		\label{Equa:tilde_eta_def}
		\int \phi(\gamma) \tilde \eta(d\gamma) = \int \phi(\gamma) {(1 - \|\pi_\gamma\|)} \eta(d\gamma) + \int \bigg[ \int \phi(\tilde \gamma'_{\gamma}) \pi_{\gamma}(d\gamma') \bigg] \eta(d\gamma). 
		\end{equation}
		The meaning of $\tilde \eta$ is that part of the curve $\gamma$ is replaced with the curve $\tilde \gamma'_{\gamma}$ constructed in \eqref{Equa:gamma_gamma_prime_map} with weight according to $\pi_{\gamma}(d\gamma')$.  $\tilde{\eta}$ is clearly a dissipative solution.
		
	Notice that $(e_0,W_0)_{\#}\tilde{\eta}=\nu_0$.

			Indeed, since
			\begin{equation*}
			\tilde \gamma'_\gamma(0) = \gamma(0), \quad \tilde \gamma'_\gamma(-1) = \gamma(-1), 
			\end{equation*}
			the claim is proved.
	
	Now let us compute the dissipation for $\tilde{\eta}$. One has that
	
	\begin{align}
	\int|W_0(\gamma)|^2\tilde{\eta}(d\gamma)-\int|\dot{\gamma}(t)|^2\tilde{\eta}(d\gamma)=\int|W_0(\gamma)|^2{\|\pi_{\gamma}\|}\eta(d\gamma)-\int{\biggl[}\int|\dot{\tilde{\gamma}}'_{\gamma}(t)|^2\pi_{\gamma}(\gamma'){\biggr]}\eta(d\gamma).
	\end{align}
	
	Since 
	\[
	|\dot{\tilde{\gamma}}'_{\gamma}(t)|^2=\frac{|W_0(\gamma)+W_0(\gamma)'|^2}{4}<\frac{|W_0(\gamma)|^2}{2}+\frac{|W_0(\gamma')|^2}{2}\quad \forall\,t\geq t_{\gamma,\gamma'},
	\]
	
	one has that $D(\tilde{\eta})>0$.
	
	In particular, under the assumptions of the lemma, we conclude that $H$ is negligible for all admissible plans $\pi$. Thus by \eqref{Equa:kellerer_1} there are two sets $N_1,N_2 \subset \Delta$ such that $H \subset N_1 \times \Gamma \cup N_2 \times \Gamma$. Removing $N_1 \cup N_2$ from $\Delta$ we obtain that the remaining trajectories are disjoint.
\end{proof}

\begin{remark}
	\label{Rem:converse_not_intersecting_not_true}
	The fact that there exists a dissipative solution $\eta$ with $D(\eta)=0$ does not imply that $(e_0,W_0)_{\#}\eta\in D_0$. Indeed, the Example 3 of Bressan and Nguyen \cite{BN} consists in constructing a sequence of particles whose mass is decreasing such that the intersection with of the $i$-th and $i+1$-th occurs only if the intersection of the $i+1$-th with $i+2$-th occurs. In this example $\eta$ is concentrated on a set of trajectories which if prolonged have empty pairwise intersection, but since the perturbation needed in order to make them to intersect becomes negligible also the stricly dissipating solution is a solution.
\end{remark}

Thanks to the above Propositions, it follows that for a dense $G_\delta$ set of initial data $\varpi_t$ is a measure concentrated on a graph, i.e. $\varpi_t = (u_t)_\sharp \rho_t$, with $u_t \in \L^2(\rho_t)$, and the distribution  $\pi = 0$ up to divergence free distributions. Moreover, it is fairly easy to see that for every curve $\gamma(t) = x_0 + v_0 t$ one has
\begin{equation*}
\int \phi(t,\gamma(t)) dt \leq \|\phi\|_{C^0} \frac{2 \diam(\supp(\phi))}{1 + |v_0|},  
\end{equation*}
so that
\begin{equation*}
\begin{split}
\int \phi(t,x) |u_t|^3 dxdt &= \int \bigg[ \int \phi(t,\gamma(t)) dt \bigg] \eta(d\gamma) \\
&\leq 2 \|\phi\|_{C_0} \diam(\supp(\phi)) \int \frac{|W_0(\gamma)|^3}{1 + |W_0(\gamma)|} \eta(d\gamma) < \infty.
\end{split}
\end{equation*}
Thus we can compute also the third moment of $\varpi_t$ and it belongs to $L^3_\loc(\R^+\times\R^d) $, resulting into the complete pressureless Euler system (without Young measures)
\begin{equation*}
\partial_t \rho_t + \mathrm{div}_x (\rho_t u_t) = 0, \quad \partial_t (\rho_t u_t) + \mathrm{div}_x(\rho_t u_t \times u_t) = 0, \quad \partial_t (\rho |u_t|^2) + \mathrm{div}_x (\rho_t |u_t|^2 u_t) = 0.
\end{equation*}

\begin{remark}
\label{Rem:other_G_delta}
As a final observation, we note that, for a generic approximation scheme, the only requirements for the existence of a dense $G_\delta$-set as in Theorem \ref{thm:freeflow} are that 
\begin{enumerate}
\item for a dense set of initial data (e.g. finitely may $\delta$'s as in Lemma \ref{lemma:freeflowfinite}) the only solution is the free flow;
\item the scheme dissipates a convex entropy (e.g. $D(\eta)$ is our case), l.s.c. w.r.t. to convergence of the initial data;
\item if the dissipation is $0$, the only solution is the free flow.
\end{enumerate}
We observe the intersection of dense $G_\delta$-sets is a dense $G_\delta$ set, so that we scan say that up to a set of first category the solution contructed by this scheme coincide with the dissipative solution of Definition \ref{def:sticky}.
\end{remark}

	\end{document}